\documentclass[12pt]{amsart}
\usepackage{amssymb,latexsym,eucal}

\usepackage{hyperref}
\usepackage[letterpaper,centering,text={6.5in,9in}]{geometry}
\usepackage{graphicx}
\usepackage{breqn}

\theoremstyle{plain}
\newtheorem{theorem}{Theorem}

\newtheorem{lemma}{Lemma}
\newtheorem{proposition}{Proposition}
\newtheorem{question}{Question}
\theoremstyle{definition}
\newtheorem{definition}{Definition}

\newtheorem{assumption}{Assumption}
\theoremstyle{remark}
\newtheorem{remark}{Remark}

\newcommand{\numbersystem}{\mathbb}
\newcommand{\N}{\numbersystem{N}}

\newcommand{\R}{\numbersystem{R}}
\newcommand{\C}{\numbersystem{C}}

\newcommand{\CP}{\numbersystem{CP}}
\newcommand{\PP}{\numbersystem{P}}

\begin{document}

\title{J-holomorphic cylinders between ellipsoids in dimension four}
\author{Richard K. Hind}
\address{RH: Department of Mathematics\\
University of Notre Dame\\
255 Hurley\\
Notre Dame, IN 46556, USA.}
\author{Ely Kerman}
\address{EK: Department of Mathematics\\
University of Illinois at Urbana-Champaign\\
1409 West Green Street\\
Urbana, IL 61801, USA.}
\thanks{Both authors are supported by grants from the Simons Foundation}

\begin{abstract}
We establish the existence of regular $J$-holomorphic cylinders in $4$-dimensional ellipsoid cobordisms that are asymptotic to Reeb orbits of small Conley-Zehnder index. We also give an independent proof of the existence of cylindrical holomorphic buildings in $4$-dimensional ellipsoid cobordisms under fairly general conditions.

\end{abstract}
\maketitle

\section{Introduction}

In this note we establish the existence of regular pseudoholomorphic cylinders between Reeb orbits of the same, suitably small, Conley-Zehnder index in ellipsoid cobordisms. The existence of holomorphic {\it buildings} between such orbits follows from the construction of cylindrical contact homology, see \cite{egh} and \cite{pa}. Genuine holomorphic curves though are essential for various applications since they are amenable to gluing. In particular the cylinders constructed here are used to correct an error in the main existence result from \cite{hk}, see \cite{hk2}. They are also a key ingredient in McDuff's proof of stabilized symplectic embedding obstructions -- Theorem 1.1 in \cite{mcduffpre}, see Lemma 2.5 in the same article. It would be easy to conjecture that holomorphic cylinders between Reeb orbits of the same Conley-Zehnder index should always exist in ellipsoid cobordisms. However this is far from the case, as we discuss in Section \ref{results}.

\subsection{The Setting} We begin by recalling the features of a standard symplectic ellipsoid in $\R^4=\C^2$. To a pair of real positive numbers $a<b$  which are rationally independent we associate the symplectic ellipsoid
 $$E=E(a,b) =\left\{ \frac{\pi|z_1|^2}{a} + \frac{\pi |z_2|^2}{b}  \leq 1\right\}.$$ The notation $cE(a,b)$ will be used for the scaled ellipsoid $E(ca,cb)$.

We denote the  boundary of $E$ by $\partial E$ and its interior by $\mathring{E}$. The restriction of the standard Liouville form on $\R^4$ yields a contact form  $\lambda =\lambda(a,b)$ on $\partial E$. It has exactly two simple closed Reeb orbits, a faster orbit   $\alpha$ with period $a$ and a slower orbit $\beta$ with period $b$. The Conley-Zehnder index of $\alpha^k$, the $k$-th iterate of $\alpha$,  is
\begin{equation}\label{czsmall}
\mathrm{CZ}(\alpha^k) =2k+ 2\left\lfloor \frac{k a}{b} \right\rfloor + 1
\end{equation}
and the the Conley-Zehnder index of the $k$-th iterate of $\beta$ is
\begin{equation}
\mathrm{CZ}(\beta^k) =2k+ 2\left\lfloor \frac{k b}{a} \right\rfloor + 1.
\end{equation}
We also note that for every natural number  $k$ exactly one iterate of either $\alpha$ or $\beta$ has Conley-Zehnder index equal to $2k+1$.

Now consider a pair of nested ellipsoids, $E_1=E(a_1, b_1)$ and $E_2=E(a_2, b_2)$ with
$a_1 < a_2$ and $b_1<b_2$. The set $E_2 \smallsetminus \mathring{E_1}$ equipped with the symplectic form inherited from $\R^4$ is a compact symplectic cobordism from the contact manifold  $(\partial E_1, \lambda_1)$ to the contact manifold $(\partial E_2, \lambda_2)$ where $\lambda_i= \lambda(a_i,b_i)$.  Let $(X_1^2, \omega)$ be its symplectic completion.
An almost complex structure on
$(X_1^2, \omega)$ is said to be \textit{admissible} if it is tamed by $\omega$ and is (eventually) cylindrical on the two ends of the symplectic completion.
\subsection{The Results}\label{results}

We consider the following basic question.

\begin{question}\label{all k}
Let $J$ be an admissible almost complex structure $J$ on $(X_1^2, \omega)$.
For which natural numbers $k$ does there exists a regular $J$-holomorphic cylinder whose ends are asymptotic to the unique closed Reeb orbits of $\lambda_1$ and $\lambda_2$ with Conley-Zehnder index equal to $2k+1$?
\end{question}

\noindent{\bf Answer 1.} {\it Such cylinders do not exist for all $k$.} The application we have in mind for these cylinders are  symplectic embedding obstructions, see \cite{hk2} and \cite{mcduffpre}. If such cylinders did exist for all $k$, then  we could readily obstruct embeddings which are known to exist! For example, it was shown in \cite{ch} that there exist regular genus $0$ curves in the symplectic completion of the cobordism $E(c,c+\epsilon) \smallsetminus \mathring{E}(1,\frac{13}{2}+\delta)$ with five positive ends each asymptotic to $\beta_2$ and a single negative end asymptotic to $\alpha_1^{13}$. This is true for any $c$ large enough to enable such a cobordism. Now consider the symplectic completion of the  cobordism $E(1,\frac{13}{2}+\delta) \smallsetminus \mathring{E}(\lambda, 4\lambda + \delta)$, where $\lambda$ is chosen sufficiently small to guarantee a cobordism. In this cobordism we observe that $\mathrm{CZ}(\alpha_2^{13}) = \mathrm{CZ}(\alpha_1^{12})$. If we had a regular holomorphic cylinder corresponding to $k = \mathrm{CZ}(\alpha_1^{12})$ then, by a gluing theorem and rescaling argument, we could construct genus $0$ curves in the symplectic completion of the cobordism $E(c,c+\epsilon) \smallsetminus \mathring{E}(1,4+\delta)$ with five positive ends each asymptotic to $\beta_2$ and a single negative end asymptotic to $\alpha_1^{12}$. A compactness theorem in \cite{ch}, see also \cite{hk}, implies that such curves will exist whenever we can form such a cobordism from a symplectic embedding $E(1,4) \hookrightarrow E(c,c+\epsilon)$. The fact that the symplectic area of these curves is positive implies that $5c \ge 12$. This contradicts the existence of a symplectic embedding $E(1,4) \hookrightarrow E(2,2+\epsilon)$, see \cite{opshtein}, \cite{ms}.

More general obstructions to the existence of cylinders come from Embedded Contact Homology (ECH), see \cite{hutlec}. This is discussed by McDuff in \cite{mcduffpre}, Remark 2.8, where a specific example appears. For a sequence of such examples we offer the following. Let $g_n$ be the $n^{th}$ odd index Fibonacci number. These numbers are determined by the conditions $g_0=g_1=1$ and the recursion relation
$
g_{n+2}=3g_{n+1} -g_n.
$
We say that an admissible almost complex structure $J$ on $(X_1^2, \omega)$ is \textit{simply generic} if every somewhere injective $J$-holomorphic curve of genus zero is regular. It follows from \cite{dr}, that the  simply generic $J$'s form a residual subset of all smooth and admissible almost complex structures tamed by $\omega$.

\begin{proposition}\label{nope}
Fix a natural number $n\geq 2$ such that $g_{n+1}$ is odd.  For a positive real number $c_1<1$ and an irrational positive number $\epsilon>0$ consider the nested pair of irrational ellipsoids
$$E_1 =c_1E\left(1, \frac{g_{n+2}-g_n}{2g_n}+\epsilon \right) \subset E_2=E\left(1, \frac{g_{n+2}}{g_n}+\epsilon\right),$$
and denote their faster simple  closed Reeb orbits by $\alpha_1$ and $\alpha_2$, respectively.

If  $\epsilon>0$ is sufficiently small then the orbits $\alpha_1^{g_{n+2}-g_n}$
and $\alpha_2^{g_{n+2}}$ have the same Conley-Zehnder index. If, in addition, $J$ is simply generic and $c_1$ is sufficiently close to $1$ then there are no (regular) $J$-holomorphic cylinders from  $\alpha_1^{g_{n+2}-g_n}$ to  $\alpha_2^{g_{n+2}}$.
\end{proposition}

As pointed out to us by Dusa McDuff, the proof of this not difficult given the ECH machinery and so it is omitted.

\medskip

\noindent{\bf Answer 2.} {\it Cylindrical buildings do exist for all $k$.}
If one relaxes Question \ref{all k} by asking only for cylindrical $J$-holomorphic buildings, then these do exist for all $k$.
This follows from Pardon's proof, in \cite{pa}, of the invariance of contact homology.

\begin{theorem}\label{pardon}(\cite{pa})
For any admissible almost complex structure $J$ on $(X_1^2, \omega)$,  there exists for each $k \in \N$ a cylindrical\footnote{The curves of a holomorphic building in $X_1^2$ have compactifications and these glue together to yield a continuous map to $\bar{X}$. The building is said to be cylindrical (or planar) if the domain of this continuous map is a cylinder (plane).} $J$-holomorphic building $\mathbf{ H}_k$ whose ends are asymptotic to the unique closed Reeb orbits of $\lambda_1$ and $\lambda_2$ with Conley-Zehnder index equal to $2k+1$.
\end{theorem}

\medskip

\noindent{\bf Answer 3.} {\it For certain pairs of ellipsoids and sufficiently small values of $k$, regular cylinders do exist.}

\medskip
\noindent This is the content of our main result.

\begin{theorem}\label{main} Suppose that $E_1=E(a_1, b_1)\subset E_2=E(a_2, b_2)$ is a nested pair of irrational ellipsoids and denote their faster simple closed Reeb orbits by $\alpha_1$ and $\alpha_2$, respectively. If $k\in \N$ satisfies
\begin{equation*}\label{<b}
k<\frac{b_1}{a_1}
\end{equation*}
and the unique closed Reeb orbit of $\lambda_2 = \lambda(a_2, b_2)$ with Conley-Zehnder index equal to $\mathrm{CZ}(\alpha_1^k)$ is an iterate, $\alpha_2^{\ell}$, of the orbit $\alpha_2$,  then for a generic choice of admissible almost complex structure $J$ on $(X_1^2, \omega)$ there exists  a regular $J$-holomorphic cylinder $u_k$ from $\alpha_1^k$  to $\alpha_2^\ell$.
\end{theorem}

Our proof of Theorem \ref{main} starts with the existence of holomorphic buildings, as in Theorem \ref{pardon}. For completeness, we include an alternative proof of Theorem \ref{pardon}, albeit under an additional assumption, see Theorem \ref{alt}. We hope this is of independent interest. The assumption of Theorem \ref{alt} is weaker than the hypotheses assumed in Theorem \ref{main}, so altogether we will give a self-contained proof of Theorem \ref{main}.

\subsection{Structure} The proof of Theorem \ref{main}, given Theorem \ref{pardon},  is presented in Section \ref{proof given Pardon}. In Section \ref{alternate}, we present our alternative proof of Theorem \ref{alt}.

\subsection{Acknowledgement.} The authors thank Dusa McDuff for reading an early draft and for many very useful conversations on these topics.


\section{Proof of Theorem \ref{main}}\label{proof given Pardon}

\subsection{The Plan}
Our proof of Theorem \ref{main} is divided into two propositions.

\begin{proposition}\label{I}
Let $J$ be an admissible almost complex structure on $(X_1^2, \omega)$ which is simply generic. Then there is a $J$-holomorphic cylinder $u_k$ from $\alpha_1^k$  to $\alpha_2^\ell$.
\end{proposition}

Given Proposition \ref{I} it remains to prove that we can also, generically, ensure that $u_k$ is regular. To achieve this we will use two results of Wendl from \cite{we}. The first of these is the following implication of Corollary 3.17 in \cite{we}.
\begin{theorem}\label{nocrit}(\cite{we})
For generic admissible $J$ on $(X_1^2, \omega)$, every somewhere injective $J$-holomorphic curve of index zero is immersed.
\end{theorem}
With this we prove the following.

\begin{proposition}\label{II}
Suppose that $J$ is an admissible almost complex structure on $(X_1^2, \omega) $ which is simply generic and is also generic in the sense of Theorem \ref{nocrit}. Then the $J$-holomorphic cylinder $u_k$ in the statement of Proposition \ref{I} is regular.
\end{proposition}

Theorem \ref{main}  follows easily from Propositions \ref{I} and \ref{II}.

\subsection{Index formula}
We recall here the formula for the index of curves in an exact cobordism between ellipsoids. For two irrational ellipsoids $E_1$ and $E_2$ let $(X, \omega)$ be the completion of topologically trivial and exact
symplectic cobordism from the contact manifold $\partial E_1$ to the contact manifold $\partial E_2$. Let $J$ be an admissible  almost complex structure on $(X, \omega)$ and suppose that  $u$ is a $J$-holomorphic curve of genus zero such that  $u$ has
\begin{itemize}
\item $m^+$ positive punctures asymptotic to $\alpha_2$ with multiplicities $r^+_1, \dots , r^+_{m^+}$,
\item $n^+$ positive punctures asymptotic to $\beta_2$ with multiplicities $s^+_1, \dots , s^+_{n^-+}$,
\item $m^-$ negative punctures asymptotic to $\alpha_1$ with multiplicities $r^-_1, \dots , r^-_{m^-}$,
\item and $n^-$ negative punctures asymptotic to $\beta_1$ with multiplicities $s^-_1, \dots , s^-_{n^-}$.
\end{itemize}
The index of $u$ is then given by
\begin{dmath}\label{standard}
\mathrm{index}(u) = m^+ + n^+ + m^- + n^- -2 + \sum_1^{m^+} \mathrm{CZ}(\alpha_2^{r^+_i}) + \sum_1^{n^+} \mathrm{CZ}(\beta_2^{s^+_i})- \sum_1^{m^-} \mathrm{CZ}(\alpha_1^{r^-_i}) - \sum_1^{n^-} \mathrm{CZ}(\beta_1^{s^-_i})= 2\left[  m^+ + n^+ -1 + \sum_1^{m^+}\left(r^+_i + \left\lfloor \frac{r^+_ia_2}{b_2} \right\rfloor \right) + \sum_1^{n^+}\left(s^+_i + \left\lfloor \frac{s^+_ib_2}{a_2} \right\rfloor \right)-\sum_1^{m^-}\left(r^-_i + \left\lfloor \frac{r^-_ia_1}{b_1} \right\rfloor \right) -  \sum_1^{n^-}\left(s^-_i + \left\lfloor \frac{s^-_ib_1}{a_1} \right\rfloor \right)\right].
\end{dmath}

\subsection{Proof of Proposition \ref{I}}
Given an ellipsoid $E= E(a,b)$,  we first establish some simple properties of pseudoholomorphic  curves mapping to  the symplectization of its boundary,  $$(Y, d(e^{\tau} \lambda))=(\R \times \partial E(a,b), d(e^{\tau} \lambda(a,b))).$$
Recall that an almost complex structure $J$ on $Y$ is said to be $\lambda$-cylindrical if the restriction of $J$ to $\ker \lambda$ is compatible with $d\lambda$ and $J(\partial_{\tau})$ is equal to the Reeb vector field of $\lambda$.

\begin{lemma}\label{most}
Suppose that $J$ is a $\lambda$-cylindrical almost complex structure on $Y$ and that  $v$ is a $J$-holomorphic curve  with exactly one positive puncture. The index of $v$ is even and nonnegative,  and is at least two unless $v$  covers a trivial cylinder. Moreover,
if the positive puncture is asymptotic to the closed Reeb orbit $\eta$ of $\lambda$, then for any negative puncture of $v$ asymptotic to some $\gamma$, we have
\begin{equation*}
\mathrm{CZ}(\gamma) \leq \mathrm{CZ}(\eta).
\end{equation*}

\end{lemma}

\begin{proof}
The positive puncture of $v$ is asymptotic to an iterate of either $\alpha$ or $\beta$. The proofs in both cases are essentially the same. So, we  assume that the positive puncture of $v$ is asymptotic to $\eta =\alpha^{r^+}$  for some $r^+$, and leave the remaining case to the reader.

Suppose that  the curve $v$ has  $m^-$ negative punctures asymptotic to $\alpha$ with multiplicities $r^-_1, \dots , r^-_{m^-}$, and $n^-$ negative punctures asymptotic to $\beta$ with multiplicities $s^-_1, \dots , s^-_{n^-}$. Denoting the domain of $v$ by $\mathcal{D}$, we have
\begin{equation}
\label{aa}
0\leq \int_{\mathcal{D}} v^*(d\lambda) = r^+a- \sum_1^{m^-} r^-_ia -\sum_1^{n^-} s^-_i b.
\end{equation}
From this we see that $r^+a$, the period of  $\eta$, is at least as large at the period of any closed Reeb orbit corresponding to a negative puncture of $v$. The second assertion above then follows from the fact that for irrational ellipsoids the orderings of closed Reeb orbits by period and by Conley-Zehnder index coincide.

To prove the first assertion we note that  \eqref{standard} implies that the index of $v$ is  given by
\begin{equation}
\label{eindex}
\mathrm{index}(v) =2\left( r^+ + \left\lfloor \frac{r^+a}{b} \right\rfloor  -\sum_1^{m^-}\left(r^-_i + \left\lfloor \frac{r^-_ia}{b} \right\rfloor \right) -  \sum_1^{n^-}\left(s^-_i + \left\lfloor \frac{s^-_ib}{a} \right\rfloor \right)\right).
\end{equation}
From \eqref{aa} we have
\begin{equation}
\label{part1}
r^+- \sum_1^{m^-} r^-_i -\sum_1^{n^-} \left\lfloor \frac{s^-_ib}{a} \right\rfloor \geq 0
\end{equation}
and
\begin{equation*}
\label{nofloor}
  \frac{r^+a}{b} - \sum_1^{m^-} \left\lfloor \frac{r^-_ia}{b} \right\rfloor -\sum_1^{n^-}  s^-_i \geq 0
\end{equation*}
which, in turn, implies that
\begin{equation}
\label{part2}
 \left\lfloor \frac{r^+a}{b} \right\rfloor - \sum_1^{m^-} \left\lfloor \frac{r^-_ia}{b} \right\rfloor -\sum_1^{n^-}  s^-_i \geq 0.
 \end{equation}
Together with \eqref{eindex}, inequalities \eqref{part1} and \eqref{part2} imply that $\mathrm{index}(v)$ is even and nonnegative. Finally we note that the index of $v$ is zero only if $ \int_{\mathcal{D}} v^*(d\lambda) =0$ in which case $v$ must cover the trivial cylinder over $\alpha$.
\end{proof}

%

Now, we return to the setting of Theorem \ref{main} where we have nested ellipsoids  $E_1=E(a_1, b_1)$ and $E_2=E(a_2, b_2)$, the symplectic completion $(X_1^2, \omega)$ of $E_2 \smallsetminus \mathring{E_1}$, and positive integers $k$ and $\ell$ such that
\begin{equation}\label{<b}
k<\frac{b_1}{a_1}
\end{equation}
and $\mathrm{CZ}(\alpha_1^k)=\mathrm{CZ}(\alpha_2^\ell)$. Condition \eqref{<b} implies that $\mathrm{CZ}(\alpha_1^k) =2k+1$ and $\mathrm{CZ}(\alpha_1^k)=\mathrm{CZ}(\alpha_2^\ell)$  if and only if
\begin{equation}
\label{lk}
k = \ell + \left\lfloor \frac{\ell a_2}{b_2} \right\rfloor.
\end{equation}

By Theorem \ref{pardon} (as well as Theorem \ref{alt} below) there is a cylindrical $J$-holomorphic building $\mathbf{ H} = \mathbf{ H}_k$ in $X_1^2$ of index zero which runs from  $\alpha_1^k$  to $\alpha_2^\ell$.
The curves of $\mathbf{H}$ map to either $Y_1$, $X_1^2$ or $Y_2$ where $Y_j$ is the symplectization of $ \partial E_j$. We now prove that the curves of $\mathbf{H}$ that map to these symplectizations must cover trivial cylinders. Hence $\mathbf{ H}$ has a unique curve mapping to $X_1^2$, the desired cylinder $u_k$.

\begin{lemma}\label{oneup}
Each curve of $\mathbf{ H}$  has exactly one positive puncture.
\end{lemma}

\begin{proof}
Curves with no positive punctures are precluded by the maximal principle in the symplectizations $Y_1$ and $Y_2$, and by Stokes' Theorem in  $(X_1^2, \omega)$. Curves with more than one positive punctures are precluded by the previous fact together with the fact that the building $\mathbf{ H}$ is cylindrical.
\end{proof}

\begin{lemma}\label{positive}
Each curve of $\mathbf{ H}$ with image in $X_1^2$ has a nonnegative index.
\end{lemma}

\begin{proof} Let $u$ be a curve of $\mathbf{ H}$ with image in $X_1^2$. By Lemma
 \ref{oneup},  $u$ has a single positive puncture  asymptotic to a closed Reeb orbit $\gamma$ of $\lambda_2$.
There are two cases to consider,  either $\gamma$ is an iteration of $\alpha_2$ or it is an iteration of $\beta_2$.  Applying the second assertion of Lemma \ref{most}  to the curves of $\mathbf{ H}$ with level above that of $u$, we see that in either case we have
 \begin{equation}
\label{less}
\mathrm{CZ}(\gamma) \leq \mathrm{CZ}(\alpha_2^\ell).
\end{equation}

 \medskip
 \noindent {\bf Case 1.} \emph{$\gamma$ is an iterate of $\alpha_2$.}
 \medskip

The curve $u$ is a (possibly trivial) multiple cover of  a somewhere injective $J$-holomorphic curve, $v$.
Denote the degree of the covering by $p \in \N$.  In the present case, the positive puncture of $v$ is asymptotic to  $\alpha_2^r$ for some $r \in \N$ and,  by \eqref{less}, we have
\begin{equation}
\label{pr}
pr \leq \ell.
\end{equation}
Since  $J$ is simply generic, by assumption, the index of curve  $v$ is nonnegative. If $v$ has $m^-$ negative punctures asymptotic to $\alpha_1$ with multiplicities $r^-_1, \dots , r^-_{m^-}$, and $n^-$ negative punctures asymptotic to $\beta_1$ with multiplicities $s^-_1, \dots , s^-_{n^-}$, this implies that
\begin{equation}\label{posreg}
 r + \left\lfloor \frac{ra_2}{b_2} \right\rfloor   -\sum_1^{m^-}\left(r^-_i + \left\lfloor \frac{r^-_ia_1}{b_1} \right\rfloor \right) -  \sum_1^{n^-}\left(s^-_i + \left\lfloor \frac{s^-_ib_1}{a_1} \right\rfloor \right)  \geq 0.
\end{equation}
Invoking hypotheses \eqref{<b} and \eqref{lk} together with \eqref{pr}
we also have
 \begin{equation}
\label{noroom}
r + \left\lfloor \frac{ra_2}{b_2} \right\rfloor \leq \ell + \left\lfloor \frac{\ell a_2}{b_2} \right\rfloor = k < \frac{b_1}{a_1}.
\end{equation}
It then follows from \eqref{posreg} and \eqref{noroom}, that  $v$ (and hence $u$) can have no negative punctures asymptotic to multiples of $\beta_1$. Thus, \eqref{posreg} simplifies to
\begin{equation}\label{goto}
  r + \left\lfloor \frac{ra_2}{b_2} \right\rfloor   -\sum_1^{m^-}\left(r^-_i + \left\lfloor \frac{r^-_ia_1}{b_1} \right\rfloor \right) \geq 0.
\end{equation}

Suppose that the sum  $\sum_1^{m^-}r^-_i$ was greater than $\frac{k}{p}$. Then  inequality  \eqref{goto} would imply
\begin{equation*}
r + \left\lfloor \frac{ra_2}{b_2} \right\rfloor >\frac{k}{p}
\end{equation*}
which, together with \eqref{pr}, would yield
\begin{equation*}
\ell + \left\lfloor \frac{\ell a_2}{b_2} \right\rfloor  \geq pr + \left\lfloor \frac{pr a_2}{b_2} \right\rfloor \geq p\left(r + \left\lfloor \frac{r a_2}{b_2} \right\rfloor\right) > k.
\end{equation*}
This would  contradict the index equality \eqref{lk}.  Hence, we must instead have
\begin{equation*}
\sum_1^{m^-}r^-_i  \leq \frac{k}{p}.
\end{equation*}
This implies that
\begin{equation}
\label{left}
pr_i^-<\frac{b_1}{a_1} \text{ for all } i=1, \dots, m^-
\end{equation}
and so \eqref{goto} simplifies to
\begin{equation}
\label{right}
 r + \left\lfloor \frac{ra_2}{b_2} \right\rfloor   -\sum_1^{m^-}r^-_i  \geq 0.
\end{equation}
Invoking inequalities \eqref{left} and \eqref{right}, in succession we then get
\begin{eqnarray*}
\mathrm{index}(u) & = & 2\left( pr + \left\lfloor \frac{pr a_2}{b_2} \right\rfloor   -\sum_1^{m^-}\left(pr^-_i + \left\lfloor \frac{pr^-_ia_1}{b_1} \right\rfloor \right) \right) \\
{} & = & 2\left( pr + \left\lfloor \frac{pra_2}{b_2} \right\rfloor   -\sum_1^{m^-} pr^-_i  \right) \\
{} & \geq & 2p\left( r + \left\lfloor \frac{ra_2}{b_2} \right\rfloor   -\sum_1^{m^-}r^-_i  \right) \\
{} & \geq & 0,
\end{eqnarray*}
as desired.

 \medskip
 \noindent {\bf Case 2.} \emph{$\gamma$ is an iterate of $\beta_2$.}
 \medskip

Define $p$ and $v$ as in the previous case. The positive puncture of $v$ is now asymptotic to $\beta_2^r$ for some $r \in \N$.  We label the negative punctures of $v$ as before. The fact that the index of $v$ is nonnegative now implies that
\begin{equation*}
 r + \left\lfloor \frac{rb_2}{a_2} \right\rfloor   -\sum_1^{m^-}\left(r^-_i + \left\lfloor \frac{r^-_ia_1}{b_1} \right\rfloor \right) -  \sum_1^{n^-}\left(s^-_i + \left\lfloor \frac{s^-_ib_1}{a_1} \right\rfloor \right) \geq 0.
\end{equation*}
the second assertion of Lemma \ref{most} implies that $\mathrm{CZ}(\beta_2^{pr}) \leq \mathrm{CZ}(\alpha_2^\ell)$  and hence
\begin{equation*}
pr + \left\lfloor \frac{prb_2}{a_2} \right\rfloor \leq \ell + \left\lfloor \frac{\ell a_2}{b_2}  \right\rfloor.
\end{equation*}
As in the previous case, this together with \eqref{<b} and \eqref{lk} precludes $v$ from having  negative punctures asymptotic to multiples of $\beta_1$. So,  we now have
\begin{equation}\label{go2}
  r + \left\lfloor \frac{rb_2}{a_2} \right\rfloor   -\sum_1^{m^-}\left(r^-_i + \left\lfloor \frac{r^-_ia_1}{b_1} \right\rfloor \right)   \geq 0.
\end{equation}

If the inequality  $\sum_1^{m^-} r^-_i > \frac{k}{p}$ held,  then \eqref{go2}  would yield
\begin{equation*}
 r + \left\lfloor \frac{rb_2}{a_2} \right\rfloor  > \frac{k}{p}.
\end{equation*}
and imply that
\begin{equation*}
\ell + \left\lfloor \frac{\ell a_2}{b_2} \right\rfloor  \geq pr + \left\lfloor \frac{pr b_2}{a_2} \right\rfloor \geq p\left(r + \left\lfloor \frac{r b_2}{a_2} \right\rfloor\right) > k.
\end{equation*}
Since this contradicts \eqref{lk}  we must have
$ \sum_1^{m^-} r^-_i \leq \frac{k}{p}$ instead. From this it follows that
\begin{equation*}
pr_i^-<\frac{b_1}{a_1} \text{ for all } i=1, \dots, m^-
\end{equation*}
and
\begin{equation*}
 r + \left\lfloor \frac{rb_2}{a_2} \right\rfloor   -\sum_1^{m^-}r^-_i  \geq 0.
\end{equation*}
Thus,
\begin{eqnarray*}
\mathrm{index}(u) & = & 2\left( pr + \left\lfloor \frac{prb_2}{a_2} \right\rfloor   -\sum_1^{m^-}\left(pr^-_i + \left\lfloor \frac{pr^-_ia_1}{b_1} \right\rfloor \right) \right) \\
{} & \geq & 2p\left( r + \left\lfloor \frac{rb_2}{a_2} \right\rfloor   -\sum_1^{m^-}r^-_i \right) \\
{} & \geq & 0,
\end{eqnarray*}
again as desired
\end{proof}

At this point we can complete the proof of Proposition \ref{I}.
The total index of all the curves of $\mathbf{H}$ is zero. It therefore follows from Lemma \ref{most}, Lemma \ref{oneup} and Lemma \ref{positive} that {\bf each} curve of $\mathbf{H}$ has index zero and those curves mapping to $Y_1$ or $Y_2$ cover trivial cylinders. The compactifications of all the curves of $\mathbf{H}$  glue together to yield a continuous map $\bar{\mathbf{H}}$ from the cylinder to $E_2 \smallsetminus  \mathring{E_1}$ that connects $\alpha_1^k$ to $\alpha_2^\ell$. Since the curves mapping to  $Y_1$ or $Y_2$ cover trivial cylinders, the compactifications of each of these curves map to single closed Reeb orbits. Hence, $\mathbf{H}$ can have only one curve with image in $X_1^2$, and this curve must have  exactly one negative and one positive end asymptotic to $\alpha_1^k$ and $\alpha_2^{\ell}$, respectively. This is the desired $J$-holomorphic cylinder $u_k$.

\subsection{Proof of Proposition \ref{II}}

Assume that $J$ is simply generic and generic in the sense of Theorem \ref{nocrit}. The curve $u_k$ from Proposition \ref{I} is the $p$-fold cover of a somewhere injective $J$-holomorphic cylinder, $v$, for some $p$ in $\N$.
\begin{lemma}
The index of $v$ is zero.
\end{lemma}

\begin{proof}
The curve $v$ is a cylinder from $\alpha_1^{r}$ to $\alpha_2^{s}$ where $rp=k$ and $sp =\ell$. Invoking inequality \eqref{<b} again we get
\begin{eqnarray*}
\mathrm{index}(u_k) - p(\mathrm{index}(v)) & = & 2\left( \ell + \left\lfloor \frac{\ell a_2}{b_2} \right\rfloor   -k - \left\lfloor \frac{k a_1}{b_1} \right\rfloor  \right)- 2p\left( s + \left\lfloor \frac{sa_2}{b_2} \right\rfloor   -r - \left\lfloor \frac{ra_1}{b_1} \right\rfloor  \right) \\
{} & = & 2\left(   \left\lfloor \frac{ ps a_2}{b_2} \right\rfloor  -  p\left\lfloor \frac{ s a_2}{b_2} \right\rfloor  \right) \\
{} & \geq & 0.
\end{eqnarray*}
It follows that  the index of $v$ is at most that of $u_k$, zero. On the other hand,
by our choice of $J$, $v$ is regular and hence has nonnegative index.
Thus,  the index of $v$ must be equal to zero.
\end{proof}

By  Theorem \ref{nocrit}, the curve $v$ must be immersed. From this, it follows that the original curve $u_k$ has no critical points since both $u_k$ and $v$ are cylinders. It then follows from Theorem 1 of \cite{we}, that $u_k$ is automatically regular for any admissible $J$ on $X_1^2$.
This completes the proof of Proposition \ref{II} and hence Theorem \ref{main}.


\section{Another Path to Holomorphic Buildings}\label{alternate}


In this section we present an alternative proof of Theorem \ref{pardon} under an additional assumption which is satisfied in the situation of Theorem \ref{main}. More precisely, we prove the following.


\begin{theorem}
\label{alt}
Given ellipsoids $E_1 \subset \mathring{E}_2$, suppose that $k\in \N$ satisfies
\begin{equation*}
\mathrm{CZ}(\alpha_1^k) = \mathrm{CZ}(\alpha_1^{k-1}) +2
\end{equation*}
and that
\begin{equation*}
\mathrm{CZ}(\alpha_2^{\ell}) = \mathrm{CZ}(\alpha_1^k) \text{ for some } \ell \in \N.
\end{equation*}
Then for any admissible almost complex structure $J$ on $(X_1^2, \omega)$ there exists a cylindrical  $J$-holomorphic building $\mathbf{ H}$  of index $0$ from $\alpha_1^k$ to $\alpha_2^{\ell}$.
\end{theorem}

The compact cobordism $E_2 \smallsetminus \mathring{E_1}$ can be identified with a domain in the symplectization $(Y_1= \R \times \partial E_1, d(e^{\tau}\lambda_1))$ of the form
$$\{1 \le \tau \le h(z)\}$$
where $h: \partial E_1 \to \R_+$ is a (simple) smooth function. The hypersurface $\Sigma_2 = \{\tau = h(z)\}$ admits an open neighborhood $U_2 \subset Y_1$ which is symplectomorphic to a neighborhood of the form $\left((-\delta, +\delta) \times \partial E_2, d(e^{\tau}\lambda_2)\right) \subset Y_2$ for some $\delta>0$, where $Y_2$ is the symplectization of $\partial E_2$.

With these identifications in place, the structure of our proof of Thoerem  \ref{alt} can be described as follows. We start  by establishing the existence of a nontrivial compact class of curves in $Y_1$ which are pseudo-holomorphic with respect to a domain dependent family of almost complex structures. Then we split $Y_1$ along $\Sigma_2$ and obtain from our original class of curves, a pseudo-holomorphic building one of whose curves is a cylinder which maps to $X_1^2$ with the required asymptotics and is pseudo-holomorphic with respect to a domain dependent family of almost complex structures. This is the content of Proposition \ref{almost}. Finally we  consider a limit of such cylinders for a sequence of domain dependent families of almost complex structures that converges to a domain independent family.  This last limit gives the desired holomorphic building, $\mathbf{H}$.

\subsection{Almost complex structures}
Fix a $\tau_2>0$ such that the closure of the  neighborhood $U_2 \subset Y_1$ of the hypersurface $\Sigma_2$ is contained in $\{\tau \leq \tau_2\}=(-\infty, \tau_2) \times \partial E_1$.  Fix a $\lambda_1$-cylindrical almost complex structure $J_1$ on $Y_1$ and a $\lambda_2$-cylindrical almost complex structure $J_2$ on $Y_2$.

\begin{definition}
Let  $\mathcal{J}$ be the set of admissible almost complex structures $J$ on $Y_1$ for which there exists a compact subset $K \subset Y_1$ contained in $\{\tau \leq \tau_2\}$ such that
\begin{itemize}
  \item $J=J_1$ away from $K$.
  \item $J$ is compatible with an exact symplectic structure on $Y_1$ that  is equal to $d(e^{\tau}\lambda_1)$ away from $K$.
\end{itemize}
\end{definition}


\begin{definition}
Let  $\mathcal{J}_{\Sigma_2} \subset \mathcal{J}$ be the set of admissible almost complex structures $J$ on $Y_1$ for which there is a compact subset $K \subset Y_1$ that contains $U_2$ and is contained in $\{\tau \leq \tau_2\}$ such that
\begin{itemize}
  \item $J=J_1$ away from $K$ and is equal to $J_2$ in $U_2$, identifying $U_2$ with a subset of $Y_2$ as above.
  \item $J$ is compatible with an exact symplectic structure on $Y_1$ that  is equal to $d(e^{\tau}\lambda_1)$ away from $K$ and is equal to $d(e^{\tau}\lambda_2)$ on $U_2$.
\end{itemize}
\end{definition}


\begin{definition}
For $N \in [0, \infty)$, let $\mathcal{J}^N_{\Sigma_2}$ be the set of  almost complex structures on $Y_1$ obtained from those of $\mathcal{J}_{\Sigma_2}$ by the standard stretching procedure to a length $N$.
\end{definition}

\begin{definition}
Let $\mathcal{J}^N_{\{0,1,\infty\}}$ be the space of families of almost complex structures $J_{z \in \CP^1}$ on $ Y^N_1$ such that in disjoint open discs around $0$, $1$ and $\infty$ the complex structure $J_z$ only depends on the (local) angular coordinate .
\end{definition}

\subsection{Conventions} To simplify and clarify the arguments below we now state two conventions. The first concerns notation and the other concerns the meaning of term \textit{index}.

\medskip
\noindent \textbf{Convention 1.} Given a subset of domain independent almost complex structures $\widetilde{\mathcal{J}}$ on $Y^N_1$
the symbol $$\mathcal{J}^N_{\{0,1,\infty\}} \cap \widetilde{\mathcal{J}}$$ will be used to denote the set of families $J_z$ in  $\mathcal{J}^N_{\{0,1,\infty\}}$ such that $J_z$ belongs to $\widetilde{\mathcal{J}}$ for each $z \in \CP^1$.
\medskip

\medskip
\noindent \textbf{Convention 2.} Since we deal with domain dependent almost complex structures we must clarify what we will mean by the \textit{index} of a pseudo-holomorphic curve. Let $u$ be a pseudo-holomorphic curve in a topologically trivial and exact symplectic cobordism.  In what follows, by \textit{the index of $u$}, we will always mean the number given by the standard index formula \eqref{standard}.
This is a convenient formula because it is preserved under limits as in formulas \eqref{preserve1} and \eqref{preserve2}.
If $u$ is pseudo-holomorphic with respect to a domain independent  almost complex structure, then the index of $u$ is the virtual dimension of the moduli space of holomorphic curves containing $u$,  modulo biholomorphisms of the domain, as usual. If however  $u$  is pseudo-holomorphic with respect to a family of domain dependent  almost complex structures then the reparameterization group may have smaller dimension. In this case,  the dimension of the moduli space of holomorphic curves containing $u$ will typically exceed the index of $u$ by the defect in the dimension of the reparameterization group.
\medskip

\subsection{A moduli space of curves} Fix a point $p_1$ on the short closed Reeb orbit $\alpha_1$ on $\partial E_1$. Given an $N\geq 0$ and a  $J_z$ in $\mathcal{J}^N_{\{0,1,\infty\}} \cap \mathcal{J}$  let $\mathcal{M}^N(J_z)$ be the space of maps
\begin{equation*}
u \colon \CP^1 \smallsetminus \{0,1,\infty \} \to Y^N_1
\end{equation*}
which satisfy \footnote{Here it is understood that $\CP^1\smallsetminus \{0,1,\infty \}$ is equipped with the standard complex form $i$ inherited from $\CP^1$. }
\begin{equation*}
du(z) \circ i = J_z \circ du(z)
\end{equation*}
for all $z \in  \CP^1 \smallsetminus \{0,1,\infty \} =\C \smallsetminus \{0,1\} $, and have the following additional properties:

\medskip
\begin{enumerate}
  \item[(u1)] $u$ has a positive puncture at $0$ which is asymptotic to $\alpha_1$.
  \item[(u2)] $u(z)$ approaches $(+\infty, p_1)$ as $z$ approaches $0$ along the positive real axis.
  \item[(u2)] $u$ has a positive puncture at $1$ which is asymptotic to $\alpha_1^{k-1}$.
  \item[(u4)] $u$ has a negative  puncture at $\infty$ which is asymptotic to $\alpha_1^{k}$.
  \item[(u5)] $u(2)$ belongs to $\{\tau_2\} \times \partial E_1 \subset Y^N_1$.
  \end{enumerate}
\medskip

It follows from  the index formula \eqref{standard} and the hypothesis of Theorem \ref{alt} that each curve $u$ in $\mathcal{M}^N(J_z)$ has index equal two according to Convention 2. Taking the constraints (u2) and (u5) into consideration, it follows that each space $\mathcal{M}^N(J_z)$ has virtual dimension equal to zero.

\subsection{Seed curves}

Consider a family $J^0_z \in \mathcal{J}_{\{0,1,\infty\}} \cap \mathcal{J}$ whose elements are all $\lambda_1$-cylindrical.
\begin{lemma}\label{seed}
For a generic choice of $J^0_z$ the space $\mathcal{M}^0(J^0_z)$ contains exactly one curve and this curve is regular.
\end{lemma}

\begin{proof}
By (u1), (u2) and (u4), every  curve  in $\mathcal{M}^0(J^0_z)$ has zero $d\lambda_1$-area. Hence each curve in $\mathcal{M}^0(J^0_z)$ covers the trivial cylinder over $\alpha_1$. In fact, this covering map is uniquely determined by the conditions (u1)-(u5). The poles and zeros of the covering map are fixed by conditions (u1), (u3) and (u4). The remaining $\C^*$ ambiguity is then resolved by  conditions (u2) and (u5).
Thus,  the space $\mathcal{M}^0(J^0_z)$ contains a single curve. Choosing $J^0_z$ to be genuinely domain dependent, we may also ensure that this curve is regular. Indeed, there are no infinitesimal deformations tangent to the trivial cylinder, and deformations in the normal direction can be generically excluded using the domain dependence of $J^0_z$.
\end{proof}

\subsection{Deformation}
Let  $J_z$ be a family in $\mathcal{J}_{\{0,1,\infty\}} \cap \mathcal{J}^N$ such that $\mathcal{M}^{N}(J_z)$ is regular.
Consider a smooth path $$s \mapsto J^s_z \in\mathcal{J}_{\{0,1,\infty\}} \cap \mathcal{J}^{sN}$$ connecting the family $J^0_z$ from Lemma \ref{seed} to $J_z$.
\begin{proposition}\label{def}
For a generic choice of the path  $J_z^{s}$,  the space
\begin{equation*}
\{ (s,u) | u \in \mathcal{M}^{sN}(J_z^s)\}
\end{equation*}
is a compact $1$-dimensional manifold.
\end{proposition}

Since the space  $\mathcal{M}^0(J^0_z)$ is regular,  the proof of the relevant transversality assertion follows from standard arguments.  Turning then to compactness, we consider a sequence $(s_n, u_n)$ in the space $\{ (s,u) | u \in \mathcal{M}^{sN}(J_z^s)\}$ such that the $s_n$ converge to some $\hat{s}>0$ and the $u_n$ converges in the sense of \cite{behwz} to a holomorphic building $\mathbf{F}$. To prove Proposition \ref{def} we must show that all but one of the curves of $\mathbf{ F}$ are trivial cylinders and the remaining (nontrivial) curve belongs to $\mathcal{M}^{\hat{s}N}(J_z^{\hat{s}})$.

We begin our analysis of the curves of  $\mathbf{ F}$ with two simple observations.
First we mention the following result whose proof is analogous to that of Lemma \ref{oneup}.
\begin{lemma}
\label{twoup}
Each curve of $\mathbf{F}$ has at least one and at most two positive punctures.
\end{lemma}
Next we note that the compactness results of \cite{behwz} imply that curves of $\mathbf{ F}$ mapping to $Y_1^{\hat{s}N}$ all lie in the same level of the building $\mathbf{F}$. We will henceforth refer to this as the \textit{$\tau_2$-level} since one of the curves at this level must satisfy constraint (u5). The remaining curves map to $Y_1$.

\medskip



Now, away from $\{0,1, \infty\}$ and a finite (possibly empty) set of additional punctures, $\{z_1, \dots, z_L\}$,  the curves $u_n$ converge in the $C^{\infty}_{\mathrm{loc}}$-topology (possibly up to translation) to a unique curve $\mathbf{C}$ of $\mathbf{ F}$. Eventually, we will show that $\mathbf{C}$ is the only nontrivial curve of $\mathbf{F}$. To this end we use $\mathbf{C}$ to organize our analysis of the other curves.
By the convergence theorem, the curves of $\mathbf{ F}$ other than $\mathbf{ C}$ can be sorted into disjoint subsets indexed by the punctures of $\mathbf{C}$. Denote by  $\eta_0$, $\eta_1$, $\eta_{\infty}$ and $\eta_{z_j}$ the closed Reeb orbits of $\lambda_1$ corresponding to the punctures of $\mathbf{C}$ at $0,1, \infty$, and $z_j$, respectively. We then define $\mathcal{C}_0$, $\mathcal{C}_1$ and $\mathcal{C}_{\infty}$ to be the cylindrical holomorphic subbuildings of $\mathbf{ F}$ asymptotic to $\eta_0$, $\eta_1$ and $\eta_{\infty}$ respectively at one end, and $\alpha_1$, $\alpha_1^{k-1}$ and $\alpha_1^k$ at the other end. Similarly $\mathcal{D}_j$ is the collection of curves which fit together to form a planar component asymptotic to $\eta_{z_j}$. By Lemma \ref{twoup} we see that $\eta_{z_j}$ must be a positive puncture for the matching curve in $\mathcal{D}_j$, and hence $z_j$ is a negative puncture for $\mathbf{C}$.

Our definition of the space $\mathcal{J}^N_{\{0,1,\infty\}}$, in particular the prescribed nature of the domain dependence near the punctures,  implies the following.

\begin{lemma}\label{s1}
Suppose that  $C$ is a curve in $\mathcal{C}_0$, $\mathcal{C}_1$ or $\mathcal{C}_{\infty}$ which maps to the $\tau_2$-level and is not part of a planar component in the complement of $\mathbf{ C}$. Then $C$ has two distinguished punctures which do not match with planar components, say at $0$ and $\infty$. Let $\Gamma$ denote the remaining punctures. The corresponding map with domain $\C^* \setminus \Gamma$ is then pseudo-holomorphic with respect to a domain dependent family of almost complex structures that  only depend on $\arg z$.
\end{lemma}

\begin{remark}\label{subtle} This subtle feature of the curves  in $\mathcal{C}_0$, $\mathcal{C}_1$ or $\mathcal{C}_{\infty}$ at  the $\tau_2$-level will play a crucial role in the proof. The precise nature of domain dependence will allow us to quantify the difference between the indices of curves and the virtual dimensions of the moduli spaces containing them (see the proof of Lemma \ref{Cinf}).
\end{remark}

The following property of $\mathbf{ C}$ will be used several times.

\begin{lemma}\label{2cases}
The curve $\mathbf{ C}$ lies at or above the $\tau_2$-level.
\end{lemma}

\begin{proof}
If $\mathbf{ C}$ was below the $\tau_2$-level then the point constraint (u5) would force $\mathbf{C}$ to have a puncture at $2$ and $\mathbf{ F}$ to include a curve in the corresponding planar component with only negative punctures. Stokes' Theorem precludes such curves.
\end{proof}

Now we begin our analysis of the collections $\mathcal{C}_0$, $\mathcal{C}_1$, $\mathcal{C}_{\infty}$, and   $\mathcal{D}_j$. We start with
the collections  $\mathcal{D}_j$ corresponding to the additional puctures of $\mathbf{C}$. Each  $\mathcal{D}_j$ is a planar building with a positive puncture asymptotic to $\eta_{z_j}$.  The index formula \eqref{standard} (which applies equally well to buildings) yields the following.
\begin{lemma}\label{D}
The sum of the indices of the curves  in ${\mathcal{ D}}_j$
is at least $2$.
\end{lemma}


Next we analyze the collections $\mathcal{C}_0$, $\mathcal{C}_1$, and $\mathcal{C}_{\infty}$, in turn.

\begin{lemma}\label{C1}
If the puncture of $\mathbf{ C}$  at $1$ is  positive, then the curves of $\mathcal{C}_1$ either all cover trivial cylinders or their collective index is at least two. If the puncture of $\mathbf{ C}$  at $1$ is negative, then the curves of $\mathcal{C}_1$ have a collective index of at least six.
\end{lemma}

\begin{proof}
The curves of $\mathcal{C}_1$ form a cylindrical building  that  connects $\eta_1$ to $\alpha_1^{k-1}$.
Let $\mathcal{C}^+_1$ be the subset of curves of $\mathcal{C}_1$  that lie above the level of $\mathbf{ C}$. By the maximum principle, the curves of  $\mathcal{C}^+_1$ comprise a connected building whose total domain is a sphere with at least two punctures. Exactly one of these punctures is positive and it is asymptotic to $\alpha_1^{k-1}$.
By Lemma \ref{2cases} the curves of  $\mathcal{C}^+_1$ are all $J_1$-holomorphic.  Arguing as in  Lemma \ref{most} it then follows that the curves in $\mathcal{C}^+_1$ either all cover a trivial cylinder or their collective index is at least two.

If the puncture of $\mathbf{C}$  at $1$ is positive, then at least one of the negative punctures of  $\mathcal{C}^+_1$ is asymptotic to $\eta_1$.  The rest of the curves of $\mathcal{C}_1$ can then be sorted into subsets indexed by the negative punctures of the building ${\mathcal{C}}^+_1$ other than the one that is asymptotic to $\eta_1$. Each of these subsets is a planar building (which  caps a negative puncture of ${\mathcal{C}}^+_1$). Arguing as in Lemma \ref{D}, it follows that  these curves also contribute at least two to the total index. This settles the first assertion of Lemma \ref{C1}.

Suppose that  the puncture of $\mathbf{C}$  at $1$ is  negative. Then the curves of $\mathcal{C}_1$
not in $\mathcal{C}^+_1$ can all be sorted into subsets indexed by the negative punctures of ${\mathcal{C}}^+_1$. Exactly one of these subsets forms a cylindrical subbuilding with two positive ends. The index formula \eqref{standard} implies that this cylindrical subbuilding contributes at least six to the total index. The other subsets form planar buildings which, as above, contribute at least 2 to the total index.
\end{proof}

\begin{lemma}\label{C0}
If the puncture of $\mathbf{C}$  at $0$ is positive, then the curves of $\mathcal{C}_0$ are all trivial. If the puncture of $\mathbf{C}$ at $0$ is negative, then the curves of $\mathcal{C}_0$ have a collective index of at least six.
\end{lemma}

\begin{proof}
Let $\mathcal{C}^+_0$ be the subset of curves of $\mathcal{C}_0$  that lie above the level of $\mathbf{ C}$. As the unmatched positive puncture of $\mathcal{C}^+_0$ is asymptotic to $\alpha_1$, which has minimal action, we see that $\mathcal{C}^+_0$ is just a trivial cylinder over $\alpha_1$.
If the puncture of $\mathbf{C}$  at $0$ is positive then there can be no other curves of $\mathcal{C}_0$ and we are done.

If the puncture of $\mathbf{C}$  at $0$ is negative then the curves of $\mathcal{C}_0$ at and below the level of $\mathbf{ C}$ form a cylindrical building with two positive ends connecting $\eta_1$ to  $\alpha_1$. Appealing again to the index formula \eqref{standard}, it follows that  these curves contribute at least six to the total index.

\end{proof}


\begin{lemma}\label{Cinf}
If $\mathbf{C}$ is in the $\tau_2$-level, then the curves of $\mathcal{C}_{\infty}$ either all cover trivial cylinders or their collective index is at least $2$. If $\mathbf{C}$ is above  the $\tau_2$-level, then $\mathcal{C}_{\infty}$ has exactly one curve, $\mathbf{ C}_{\infty}$, that is at the $\tau_2$-level and is not part of a planar component. It has index at least $-2$ and the other curves of  $\mathcal{C}_{\infty}$ either all cover trivial cylinders or their collective index is at least $2$.
\end{lemma}

\begin{proof}
When $\mathbf{C}$ is in the $\tau_2$-level the maximum principle implies that all the curves of $\mathcal{C}_{\infty}$  must lie below the $\tau_2$-level. (Otherwise $\mathcal{C}_{\infty}$ would include a curve with no positive ends.)
Thus,  each curve of $\mathcal{C}_{\infty}$ has one positive end and is pseudo-holomorphic with respect to a $\lambda_1$-cylindrical almost complex structure. The first assertion then follows as in Lemma \ref{most}.

If $\mathbf{C}$ lies above the $\tau_2$-level, then by the maximum principle $\mathcal{C}_{\infty}$ contains a unique curve, $\mathbf{C}_{\infty}$,  in the $\tau_2$-level which is not part of a planar subset.
By Lemma \ref{s1}, it can be represented as a map with domain $\C^* \setminus \Gamma$ and satisfies an $\arg z$ dependent holomorphic curve equation.
To appear as a limit in a generic $1$-parameter family we expect $\mathbf{C}_{\infty}$ to have virtual deformation index of at least $-1$, modulo reparameterization.
There is a 1-dimensional reparameterization group given by scaling, whereas domain independent curves have a 2-parameter reparameterization group given by rotation and scaling. Hence,  the virtual deformation index of $\mathbf{C}_{\infty}$ modulo reparameterization,  is one greater than its index as defined by Convention 2. So, the (conventional) index of $\mathbf{C}_{\infty}$ is at least $-2$. By Lemma \ref{most} the other curves of  $\mathcal{C}_{\infty}$ must either all cover trivial cylinders or their collective index is at least $2$.
\end{proof}

At this point we can complete the proof of Proposition \ref{def}.  Overall, we have the index equality
\begin{equation}\label{preserve1}
2= \mathrm{index}(\mathbf{F})=  \mathrm{index}(\mathbf{ C}) + \mathrm{index}(\mathcal{ C}_0)+ \mathrm{index}(\mathcal{C}_1) + \mathrm{index}(\mathcal{C}_{\infty}) + \sum_j\mathrm{index}(\mathcal{D}_j).
\end{equation}
By Lemma \ref{2cases}, $\mathbf{ C}$ is either at or above the $\tau_2$-level.


\medskip

\noindent \textbf{Assume that  $\mathbf{C}$ lies at the $\tau_2$-level.} With this,  Lemmas  \ref{D}, \ref{C1}, \ref{C0} and \ref{Cinf} and the consequence of \eqref{standard} that all indices (using our convention) are even imply that either $\mathbf{C}$ has index $2$ or its index is at most $0$. In the first case it follows from these same  Lemmas that the curves of $\mathbf{ F}$, other than $\mathbf{C}$, must all  cover trivial cylinders and so we are done.
As we now describe, the remaining case can, generically, be excluded.

Suppose that $\mathbf{C}$ has index at most $0$. The Lemmas listed above in fact imply that the index of $\mathbf{C}$ is at most  at most $-4$ if $\mathcal{C}_0$ is nontrivial.
Since $\mathbf{C}$ is at the $\tau_2$-level, it satisfies a domain dependent equation and so has a 0-dimensional reparameterization group. Hence, $\mathbf{ C}$ belongs to a moduli space whose unconstrained virtual dimension is at most $6$ when $\mathcal{C}_0$ is trivial, and at most $2$ when $\mathcal{C}_0$ is nontrivial. However, $\mathbf{C}$ has punctures at $0,1$, and $\infty$ and it must satisfy (u5). This is a 7-dimensional family of constraints and so will not generically be satisfied (for curves appearing in a 1-parameter family) if $\mathcal{C}_0$ is nontrivial. But if $\mathcal{C}_0$ is trivial,  then $\mathbf{C}$ would also have to satisfy the constraint (u2). Thus, $\mathbf{ C}$ would belong to a moduli space whose constrained virtual dimension is $-2$. Such curves can again be avoided  in generic 1-parameter families.

\medskip
\noindent \textbf{Assume that  $\mathbf{C}$ lies above the $\tau_2$-level.} We show that this situation can also be precluded. Since $\mathbf{C}$ lies above the $\tau_2$-level,  constraint (u5) implies that $\mathbf{C}$ has a puncture at $2$ in addition to its punctures at $0$, $1$, and $\infty$. By Lemma \ref{D}
we then have
\begin{equation*}
\sum_j\mathrm{index}(\mathcal{D}_j) \geq 2.
\end{equation*}
It also follows from Lemma \ref{Cinf}, that $\mathrm{index}(\mathcal{C}_{\infty}) \geq -2$.


Now, if $\mathcal{C}_0$ is nontrivial or the puncture at $1$ is negative, then Lemmas  \ref{C1} and \ref{C0} imply that $\mathrm{index}(\mathbf{ C}) \le -4$ and $\mathbf{ C}$ has a single positive end.
Taken together,  these statements contradict Lemma \ref{most}.


If $\mathcal{C}_0$ is trivial and the puncture at $1$ is positive,
then $\mathbf{ C}$ has a positive end at $0$ asymptotic to $\alpha_1$, and a positive end at $1$ of action at most that of $\alpha_1^{k-1}$. By the maximum principle there are no more positive punctures. As $\mathrm{index}(\mathcal{C}_{\infty}) \geq -2$ we see that $\mathbf{ C}$ has a negative puncture at $\infty$ of action at least that of $\alpha_1^{k-1}$. Together with the negative puncture at $2$, we observe that $\mathbf{ C}$ necessarily has nonpositive action, and indeed can only be a cover of the trivial cylinder with positive ends at $0,1$ asymptotic to $\alpha_1$ and $\alpha_1^{k-1}$ respectively, and negative punctures at $2, \infty$ asymptotic to $\alpha_1$ and $\alpha_1^{k-1}$ respectively. There is a unique such curve up to scale (translation) which satisfies constraint (u2). In particular the asymptotic approach at $\infty$ is determined, that is, approaching along the real axis we limit at a specific point on $\alpha_1$. On the other hand, by Lemma \ref{Cinf}, in a $1$-parameter family we expect to see only finitely many buildings $\mathcal{C}_{\infty}$ with asymptotes at $\alpha_1^{k-1}$ and $\alpha_1^k$, at least up to reparameterization by scaling. This determines finitely many possible asymptotic limits at the positive end and generically we do not expect the limits to coincide with that of $\mathbf{ C}$, contradicting the matching condition for limiting buildings.

\subsection{Splitting along $\Sigma_2$}
Choose a regular $J_z$ in $\mathcal{J}^0_{\{0,1,\infty\}} \cap \mathcal{J}_{\Sigma_2}$ such that the stretched families  $J^N_z$
in $\mathcal{J}^N_{\{0,1,\infty\}} \cap \mathcal{J}^N_{\Sigma_2}$ are regular for all $N \in \N$. Together,  Lemma \ref{seed} and Proposition \ref{def} imply that for all $N \in \N$ there is a curve $u_N$ in  $\mathcal{M}^N(J^N_z)$.
After passing to a subsequence,  if necessary, we may assume that the $u_N$ converge to a limiting building $\mathbf{ G}$ in the resulting split manifold.  This split manifold is diffeomorphic to $(\R \times \partial E_1) \smallsetminus \Sigma_2$ and so has two components. One is the completion $(X_1^2, \omega)$ that appears in the statement of Theorem \ref{pardon}. The other, which we denote by $(Z, \Omega)$, can be described  as the completion of the obvious symplectic cobordism from $\Sigma_2$ to $\{\tau_2\} \times  \partial E_1$. In this section we prove the following.

\begin{proposition} \label{almost}
For $J_z$ as above, there is exactly one curve of $\mathbf{ G}$ that maps to $X_1^2$. This curve, $\tilde{u}$, can be parameterized with domain $\CP^1 \smallsetminus \{0, \infty\}$ with a  positive puncture at $0$ asymptotic to $\alpha_2^{\ell}$ and a negative puncture at $\infty$ asymptotic to $\alpha_1^k$. It is pseudo-holomorphic with respect to a family of almost complex structures  $\tilde{J}_z$ on $X_1^2$ depending only on $\arg z$.
\end{proposition}

\begin{proof}

We analyze the building $\mathbf{G}$ as we analyzed $\mathbf{ F}$. The curves of $\mathbf{ G}$ map to $X_1^2$, $Z$, $Y_1$ or $Y_2$. Those mapping to $Z$ are at the same level, which we refer to again as the $\tau_2$-level. The curves $u_N$ converge to a distinguished curve $\mathbf{ C}$ of $\mathbf{ G}$ which is either at or above the $\tau_2$-level. It has punctures  at $0$, $1$ and  $\infty$, and additional punctures at the possibly empty set of points $\{z_1, \dots, z_l\}$.  We then partition the curves of $\mathbf{ G}$, other that $\mathbf{ C}$, into subsets $\mathcal{C}_0$, $\mathcal{C}_1$, $\mathcal{C}_{\infty}$ and $\mathcal{D}_j$ as before. Lemma \ref{twoup}
holds with $\mathbf{ F}$ replaced by $\mathbf{ G}$ while Lemmas \ref{D}, \ref{C1} and \ref{C0} hold as stated. Lemma \ref{s1} holds for any curves of $\mathbf{ G}$, other than $\mathbf{ C}$, which map to either $X_1^2$ or $Z$. The only significant difference involves our analysis of the subbuilding $\mathcal{C}_{\infty}$. We now have the following.

\begin{lemma}\label{Cinfnew}
The curves of $\mathcal{C}_{\infty}$ mapping to $X_1^2$ or $Z$ all have nonnegative index. The other curves of  $\mathcal{C}_{\infty}$ either all cover trivial cylinders or their collective index is at least $2$.
\end{lemma}

\begin{proof}
By Lemma \ref{s1}, the curves of $\mathcal{C}_{\infty}$  mapping to $Z$ and to $X_1^2$ and not part of a planar component are pseudo-holomorphic with respect to a domain dependent family of almost complex structures that only depends on the natural $S^1$-parameter on  their domain.  As we are no longer considering a  $1$-parameter family of almost-complex structures,  these curves are now expected to have virtual deformation index of at least $0$, modulo reparameterization.  Again,  their virtual deformation index modulo reparameterization,  is one greater than their index as defined by Convention 2.  So, the  index of each of these curves of $\mathbf{C}_{\infty}$ is at least $-1$. As the indices are all even they  are therefore all nonnegative. The other assertions concerning the curves mapping to $Y_1$ and $Y_2$ follow as before from Lemma \ref{most}.
\end{proof}

\begin{lemma}
The curve $\mathbf{ C}$ maps to the $\tau_2$-level.
\end{lemma}

\begin{proof}
Starting with
\begin{equation}\label{preserve2}
2= \mathrm{index}(\mathbf{G})=  \mathrm{index}(\mathbf{ C}) + \mathrm{index}(\mathcal{ C}_0)+ \mathrm{index}(\mathcal{C}_1) + \mathrm{index}(\mathcal{C}_{\infty}) + \sum_j\mathrm{index}(\mathcal{D}_j)
\end{equation}
we observe again that if  $\mathbf{C}$ lies above the $\tau_2$-level then constraint (u5) implies that
\begin{equation*}
\sum_j\mathrm{index}(\mathcal{D}_j) \geq 2.
\end{equation*}
This together with Lemmas \ref{C1}, \ref{C0} and \ref{Cinfnew} imply that $\mathrm{index}(\mathbf{ C}) \le 0$. Since $\mathbf{C}$ must have punctures at $0,1,2$, and $\infty$ this is a contradiction.

\end{proof}

Arguing as in the proof of Proposition \ref{def}, it follows that $\mathbf{ C}$ must have index $2$. In particular, the curves of $\mathbf{ G}$ mapping to $Y_1$ and $Y_2$ must all cover trivial cylinders , and $\mathbf{ C}$ must have exactly three punctures, two positive punctures at $0$ and $1$, asymptotic  to $\alpha_1$ and $\alpha_1^{k-1}$ respectively, and one negative puncture at $\infty$. From this it follows that the only other curve of $\mathbf{ G}$ is a cylinder mapping to $X_1^2$ with index $0$ and a negative puncture at $\infty$ asymptotic to $\alpha_1^k$. As the cylinder has index $0$ its positive puncture is asymptotic to $\alpha_2^{\ell}$.  This is the desired curve $\tilde{u}$ of Proposition \ref{almost}.

\end{proof}


\subsection{Domain Independence}

Choose a sequence $J^n_z$ of families in $\mathcal{J}^0_{\{0,1,\infty\}} \cap \mathcal{J}_{\Sigma_2}$ which are regular as in the statement of Proposition \ref{almost}, and which converge to some $J^{\infty}$ in $\mathcal{J}^0_{\{0,1,\infty\}} \cap \mathcal{J}_{\Sigma_2}$ whose domain dependence is trivial. By Proposition \ref{almost}, this sequence yields
\begin{itemize}
  \item a sequence of almost complex structures $\tilde{J}^n_{z \in\CP^1 \smallsetminus \{0, \infty\}}$ on $X_1^2$ that each only depend on $\arg z$  and which converge to a fixed almost complex structure $\tilde{J}^{\infty}$ on $X_1^2$
  \item a sequence of $\tilde{J}^n_z$-holomorphic curves $\tilde{u}_n \colon \CP^1 \smallsetminus \{0, \infty\}  \to X$ whose punctures at $0$ are positive and asymptotic to $\alpha_2^{\ell}$ and whose punctures at $\infty$ are negative and asymptotic to $\alpha_1^k$.
 \end{itemize}

After passing to a subsequence we may assume that the curves $u_n$ converge to a $\tilde{J}^{\infty}$-holomorphic cylindrical building $\mathbf{ H}$ in $X_1^2$ of index zero and genus zero with the desired asymptotics. This completes the proof of Theorem \ref{alt}.


\end{document}